\documentclass[12pt,english]{article}
\usepackage[T1]{fontenc}
\usepackage[latin9]{inputenc}
\usepackage{amsmath}
\usepackage{amssymb}

\makeatletter
\usepackage[T1]{fontenc}
\usepackage[latin9]{inputenc}
\usepackage{color}
\usepackage{amsmath}
\usepackage{graphicx}
\usepackage{amsfonts}
\usepackage{a4wide}
\usepackage{babel}%

\usepackage{mathrsfs}
\usepackage[active]{srcltx}

\providecommand{\U}[1]{\protect\rule{.1in}{.1in}}

\newtheorem{theorem}{Theorem}

\newtheorem{corollary}[theorem]{Corollary}

\newtheorem{lemma}[theorem]{Lemma}

\newtheorem{proposition}[theorem]{Proposition}
\newtheorem{remark}[theorem]{Remark}

\newenvironment{proof}[1][Proof]{\noindent\textbf{#1.} }{\ \rule{0.5em}{0.5em}}

\usepackage[usenames,dvipsnames]{pstricks}
\usepackage{epsfig}
\usepackage{pst-grad} % For gradients
\usepackage{pst-plot} % For axes

\usepackage{graphicx}

\makeatother

\usepackage{babel}
\begin{document}

\title{A priori estimates for the Fitzpatrick function }

\author{M.D. Voisei}

\date{{}}
\maketitle
\begin{abstract}
New perspectives, proofs, and some extensions of known results are
presented concerning the behavior of the Fitzpatrick function of a
monotone type operator in the general context of a locally convex
space. 
\end{abstract}

\section{Introduction and preliminaries}

Given $T:X\rightrightarrows X^{*}$ a multi-valued operator defined
in a locally convex space $X$ with valued in its topological dual
$X^{*}$, the Fitzpatrick function associated to $T$ (introduced
in \cite[Definition\ 3.1,\ p.\ 61]{MR1009594}) denoted by
\[
\varphi_{T}(x,x^{*}):=\sup\{a^{*}(x-a)+x^{*}(a)\mid(a,a^{*})\in\operatorname*{Graph}T\},\ (x,x^{*})\in X\times X^{*},
\]
is an important tool in the theory of maximal monotone operators (see
e.g. \cite{MR1009594,MR2207807,MR2577332}). 

Here $\operatorname*{Graph}T=\{(x,x^{*})\in X\times X^{*}\mid x^{*}\in T(x)\}$
is the graph of $T$, $D(T)=\operatorname*{Pr}_{X}(\operatorname*{Graph}T)$
stands for the domain of $T$, $R(T)=\operatorname*{Pr}_{X^{*}}(\operatorname*{Graph}T)=\cup_{x\in D(T)}T(x)$
denotes the range of $T$, where $\operatorname*{Pr}_{X}$, $\operatorname*{Pr}_{X^{*}}$
are the projections of $X\times X^{*}$ onto $X$, $X^{*}$, respectively.
When no confusion can occur, $T:X\rightrightarrows X^{\ast}$ is identified
with $\operatorname*{Graph}T\subset X\times X^{*}$.

Our focus in this note is on the general behavior of $\varphi_{T}$
mainly with respect to the coupling of $X\times X^{*}$ as well as
the position of the domain of $\varphi_{T}$ in $X\times X^{*}$. 

In this paper, if not otherwise explicitly mentioned, $(X,\tau)$
is a Hausdorff separated locally convex space (LCS for short), $X^{\ast}$
is its topological dual endowed with the weak-star topology $w^{*}$,
and the topological dual of $(X^{*},w^{*})$ is identified with $X$.
For $x\in X$ and $x^{*}\in X^{*}$ we set $\langle x,x^{*}\rangle:=x^{*}(x)$.

For a subset $A$ of $X$ we denote by $\operatorname*{int}A$, $\operatorname*{cl}A$
(or $\operatorname*{cl}_{\tau}A$ when we wish to underline the topology
$\tau$, or $\operatorname*{cl}_{(X,\tau)}A$ if we want to emphasize
that $A\subset(X,\tau)$), $\operatorname*{aff}A$, and $\operatorname*{conv}A$
the interior, the closure, the affine hull, and the closure hull of
$A$, respectively. If $A,B\subset X$ we set $A+B:=\{a+b\mid a\in A,\ b\in B\}$
with the convention $A+\emptyset:=\emptyset+A:=\emptyset$. 

We consider the class $\Lambda(X)$ of proper convex functions $f:X\rightarrow\overline{\mathbb{R}}:=\mathbb{R}\cup\{-\infty,+\infty\}$
and the class $\Gamma_{\tau}(X)$ (or simply $\Gamma(X)$) of those
functions $f\in\Lambda(X)$ which are $\tau$\textendash lower semi\emph{\-}continuous
(lsc for short). Recall that $f$ is \emph{proper} if $\operatorname*{dom}f:=\{x\in X\mid f(x)<\infty\}$
is nonempty and $f$ does not take the value $-\infty$.

To $f:X\rightarrow\overline{\mathbb{R}}$ we associate its \emph{convex
hull} $\operatorname*{conv}f:X\rightarrow\overline{\mathbb{R}}$ and
its ($\tau$\textendash )\emph{lsc convex hull} $\operatorname*{cl}\operatorname*{conv}f:X\rightarrow\overline{\mathbb{R}}$
($\operatorname*{cl}_{\tau}\operatorname*{conv}f$ when we want to
accentuate on the topology $\tau$) defined by 
\begin{gather*}
(\operatorname*{conv}f)(x):=\inf\{t\in\mathbb{R}\mid(x,t)\in\operatorname*{conv}(\operatorname*{epi}f)\},\\
(\operatorname*{cl}\operatorname*{conv}f)(x):=\inf\{t\in\mathbb{R}\mid(x,t)\in\operatorname*{cl}\operatorname*{conv}(\operatorname*{epi}f)\},
\end{gather*}
 where $\operatorname*{epi}f:=\{(x,t)\in X\times\mathbb{R}\mid f(x)\leq t\}$
is the \emph{epigraph} of $f$.

The \emph{conjugate} of $f:X\rightarrow\overline{\mathbb{R}}$ with
respect to the dual system $(X,X^{\ast})$ is given by 
\begin{equation}
f^{\ast}:X^{\ast}\rightarrow\overline{\mathbb{R}},\quad f^{\ast}(x^{\ast}):=\sup\{\left\langle x,x^{\ast}\right\rangle -f(x)\mid x\in X\}.\label{r6}
\end{equation}
 The conjugate $f^{\ast}$ is a weakly-star (or $w^{\ast}-$) lsc
convex function. For the proper function $f:X\rightarrow\overline{\mathbb{R}}$
we define the \emph{sub\-differential} of $f$ at $x$ by 
\[
\partial f(x):=\{x^{\ast}\in X^{\ast}\mid\left\langle x^{\prime}-x,x^{\ast}\right\rangle \leq f(x^{\prime})-f(x)\ \forall x^{\prime}\in X\},
\]
 for $x\in\operatorname*{dom}f$ and $\partial f(x):=\emptyset$ for
$x\not\in\operatorname*{dom}f$. Recall that $N_{C}=\partial\iota_{C}$
is the \emph{normal cone} of $C$, where $\iota_{C}$ is the \emph{indicator
function} of $C\subset X$ defined by $\iota_{C}(x):=0$ for $x\in C$
and $\iota_{C}(x):=\infty$ for $x\in X\setminus C$.

When $X^{\ast}$ is endowed with the topology $w^{\ast}$ (or with
any other locally convex topology $\sigma$ such that $(X^{\ast},\sigma)^{\ast}=X$),
in other words, if we take conjugates for functions defined in $X^{\ast}$
with respect to the dual system $(X^{\ast},X)$, then $f^{\ast\ast}=(f^{\ast})^{\ast}=\operatorname*{cl}\operatorname*{conv}f$
whenever $\operatorname*{cl}\operatorname*{conv}f$ (or equivalently
$f^{\ast}$) is proper.

With respect to the dual system $(X,X^{*})$, the \emph{polar of}
$A\subset X$ is $A^{\circ}:=\{x^{*}\in X^{*}\mid|\langle x,x^{*}\rangle|\le1,\ \forall x\in A\}$
while the \emph{orthogonal of} $A$ is $A^{\perp}:=\{x^{*}\in X^{*}\mid\langle x,x^{*}\rangle=0,\ \forall x\in A\}$;
similarly, the\emph{ }polar of $B\subset X^{*}$ is $B^{\circ}:=\{x\in X\mid|\langle x,x^{*}\rangle|\le1,\ \forall x^{*}\in B\}$
while the orthogonal of $B$ is $B^{\perp}:=\{x\in X\mid\langle x,x^{*}\rangle=0,\ \forall x^{*}\in B\}$. 

Let $Z:=X\times X^{\ast}$. Consider the \emph{coupling} function
\[
c:Z\rightarrow\mathbb{R},\quad c(z):=\left\langle x,x^{\ast}\right\rangle \ \text{for}\ z:=(x,x^{\ast})\in Z.
\]
 It is known that the topological dual of $(Z,\tau\times w^{\ast})$
can be (and will be) identified with $Z$ by the coupling 
\[
z\cdot z^{\prime}:=\left\langle z,z^{\prime}\right\rangle :=\left\langle x,x^{\prime\ast}\right\rangle +\left\langle x^{\prime},x^{\ast}\right\rangle \quad\text{for }z=(x,x^{\ast}),\ z^{\prime}=(x^{\prime},x^{\prime\ast})\in Z.
\]
 With respect to the natural dual system $(Z,Z)$ induced by the previous
coupling, the Fitzpatrick function of $T:X\rightrightarrows X^{*}$
has the form
\[
\varphi_{T}:Z\rightarrow\overline{\mathbb{R}},\quad\varphi_{T}(z)=\sup\{z\cdot\alpha-c(\alpha)\mid\alpha\in\operatorname*{Graph}T\},
\]
while the conjugate of $f:Z\rightarrow\overline{\mathbb{R}}$ is denoted
by 
\[
f^{\square}:Z\rightarrow\overline{\mathbb{R}},\quad f^{\square}(z)=\sup\{z\cdot z^{\prime}-f(z^{\prime})\mid z^{\prime}\in Z\},
\]
 and $f^{\square\square}=\operatorname*{cl}_{\tau\times w^{\ast}}\operatorname*{conv}f$
whenever $f^{\square}$ (or $\operatorname*{cl}_{\tau\times w^{\ast}}\operatorname*{conv}f$)
is proper. 

Note the identity
\begin{equation}
\inf_{\alpha\in\operatorname*{Graph}T}c(z-\alpha)=(c-\varphi_{T})(z).\label{eq:-5}
\end{equation}

For $E$ a LCS and $f,g:E\rightarrow\overline{\mathbb{R}}$ we set
$[f\leq g]:=\{x\in E\mid f(x)\leq g(x)\}$; the sets $[f=g]$, $[f<g]$
and $[f>g]$ are defined similarly; while $f\ge g$ means $[f\ge g]=E$
or, for every $x\in E$, $f(x)\ge g(x)$. 

\medskip

Whenever $X$ is a normed vector space, besides $X^{\ast}$ we consider
the bi-dual $X^{\ast\ast}$. We identify $X$ with the linear subspace
$J(X)$ of $X^{\ast\ast}$, where $J:X\rightarrow X^{\ast\ast}$ is
the canonical injection: $\left\langle x^{\ast},Jx\right\rangle :=\left\langle x,x^{\ast}\right\rangle $
for $x\in X$ and $x^{\ast}\in X^{\ast}$, and we denote $Jx$ by
$\widehat{x}$ or simply $x$. In this case $Z=X\times X^{\ast}$
is seen as a normed vector space with the norm $\left\Vert z\right\Vert :=\big(\left\Vert x\right\Vert ^{2}+\left\Vert x^{\ast}\right\Vert ^{2}\big)^{1/2}$
for $z:=(x,x^{\ast})$. Its topological dual $Z^{\ast}$ is identified
with $X^{\ast}\times X^{\ast\ast}$ by the coupling 
\[
\left\langle (x,x^{\ast}),(y^{\ast},y^{\ast\ast})\right\rangle :=\left\langle x,y^{\ast}\right\rangle +\left\langle x^{\ast},y^{\ast\ast}\right\rangle ,\quad(x,x^{\ast})\in X\times X^{\ast},\ (y^{\ast},y^{\ast\ast})\in X^{\ast}\times X^{\ast\ast}.
\]
 In this context, for $f:Z\rightarrow\overline{\mathbb{R}}$, the
conjugate $f^{\ast}:Z^{\ast}\rightarrow\overline{\mathbb{R}}$ is
given by (\ref{r6}) while for $f^{\square}:Z\rightarrow\overline{\mathbb{R}}$
one has
\begin{equation}
f^{\square}(x,x^{\ast})=f^{\ast}(x^{\ast},Jx)=f^{\ast}(x^{\ast},x)\quad\forall(x,x^{\ast})\in Z.\label{r11}
\end{equation}
 Note that $(f^{\square})^{\ast}(x^{\ast},x)=f^{\square\square}(x,x^{\ast})=f(x,x^{\ast})$
when $f:X\times X^{\ast}\rightarrow\overline{\mathbb{R}}$ is a proper
convex $\tau\times w^{\ast}$\textendash lsc function.

A multi\-function $T:X\rightrightarrows X^{\ast}$ is said to be
\emph{monotone} if $c(z-z^{\prime})\geq0$ for all $z,z^{\prime}\in T$
and \emph{maximal monotone} if $T$ is monotone and maximal in the
sense of inclusion. In other terms, $T$ is maximal monotone if $T$
is monotone and any element $z\in Z$ which is \emph{monotonically
related to} $T$, that is, $c(z-z^{\prime})\geq0$ for every $z^{\prime}\in T$,
belongs to $T$. The classes of monotone and maximal monotone operators
$T:X\rightrightarrows X^{\ast}$ are denoted by $\mathcal{M}(X)$
and $\mathfrak{M}(X)$, respectively. It is well known that if $f\in\Lambda(X)$
has $f\ge c$ then
\[
[f=c]=\{z\in Z\mid f(z)=c(z)\}\in\mathcal{M}(X),
\]
(see e.g. \cite[Proposition 4(h)]{MR2086060} or \cite[Lemma\ 3.1]{MR2453098}).

\eject

\section{Estimates for Negative-Infimum operators}

Recall that an operator $T:X\rightrightarrows X^{*}$ is called of
\emph{Negative-Infimum}-type or simply NI if $\varphi_{T}\ge c$ in
$Z=X\times X^{*}$. 

With respect to the natural dual system $(Z,Z)$, \emph{the support
function} of $S\subset Z$ is given by $\sigma_{S}(p):=\sup_{s\in S}s\cdot p$,
$p\in\operatorname*{dom}\sigma_{S}\subset Z$, or $\sigma_{S}=\iota_{S}^{\square}$.
For simplicity, in the sequel we use the notation $\sigma_{T-z}:=\sigma_{\operatorname*{Graph}T-z}$. 

\begin{theorem} \label{first} Let $X$ be a LCS and let $T:X\rightrightarrows X^{*}$.
Then
\begin{equation}
\forall z,p\in Z=X\times X^{*},\ t\ge0,\ (\varphi_{T}-c)(z+tp)\le(\varphi_{T}-c)(z)-t^{2}c(p)+t\sigma_{T-z}(p).\label{main}
\end{equation}

If, in addition, $T$ is NI then
\begin{equation}
\forall z,p\in Z,\ t\ge0,\ (\varphi_{T}-c)(z)\ge t^{2}c(p)-t\sigma_{T-z}(p).\label{m2}
\end{equation}

If, in addition, $T$ is NI and $z\in\operatorname*{dom}\varphi_{T}$
then $\operatorname*{dom}\sigma_{T-z}\subset[c\le0]$ and $[c=0]\subset[\sigma_{T-z}\ge0]$. 

\end{theorem}

\begin{proof} We have
\[
\forall z,p,\alpha\in Z,\ c(z+p-\alpha)=c(z-\alpha)+c(p)+p\cdot(z-\alpha).
\]
After passing to infimum over $\alpha\in\operatorname*{Graph}T$,
one gets
\[
\forall z,p\in Z,\ (c-\varphi_{T})(z+p)\ge(c-\varphi_{T})(z)+c(p)+\inf_{\alpha\in\operatorname*{Graph}T}p\cdot(z-\alpha)\Leftrightarrow
\]
\[
\forall z,p\in Z,\ (\varphi_{T}-c)(z+p)\le(\varphi_{T}-c)(z)-c(p)+\sup_{\alpha\in\operatorname*{Graph}T}p\cdot(\alpha-z)\Leftrightarrow
\]
\[
\forall z,p\in Z,\ (\varphi_{T}-c)(z+p)\le(\varphi_{T}-c)(z)-c(p)+\sigma_{T-z}(p).
\]
Change $p$ into $tp$ for $t\ge0$ to get (\ref{main}). 

The inequality (\ref{m2}) follows from $T$ being NI, more precisely
from $(\varphi_{T}-c)(z+tp)\ge0$. 

Assume, in addition, that $T$ is NI and $z\in\operatorname*{dom}\varphi_{T}$. 

If $c(p)>0$ and $\sigma_{T-z}(p)<+\infty$ then $\lim_{t\to\infty}t^{2}c(p)-t\sigma_{T-z}(p)=\infty$
in contradiction to (\ref{m2}). 

If $c(p)=0$ then for every $t\ge0$, $(\varphi_{T}-c)(z)+t\sigma_{T-z}(p)\ge0$.
Divide by $t>0$ and let $t\rightarrow+\infty$ to get $\sigma_{T-z}(p)\ge0$.
\end{proof}

\begin{remark} \label{r1} A reformulation of the last part of the
previous theorem is: let $T$ be NI and $z\in\operatorname*{dom}\varphi_{T}$.
Then
\begin{equation}
c(p)>0\Rightarrow\sigma_{T-z}(p)=+\infty,\label{eq:}
\end{equation}
\begin{equation}
c(p)=0\Rightarrow\sigma_{T-z}(p)\ge0,\label{eq:-1}
\end{equation}
\begin{equation}
\sigma_{T-z}(p)<0\Rightarrow c(p)<0.\label{eq:-2}
\end{equation}
Indeed, for (\ref{eq:-2}), if $\sigma_{T-z}(p)<0$ then $p\in\operatorname*{dom}\sigma_{T-z}$
and, from the contrapositive forms of (\ref{eq:}), (\ref{eq:-1}),
one has $c(p)\le0$ and $c(p)\neq0$, i.e., $c(p)<0$. \end{remark}

\strut

Recall that \emph{the support function} of $A\subset X$ is given
by $\sigma_{A}(x^{*})=\sup\{\langle x,x^{*}\rangle\mid x\in A\}$,
$x^{*}\in X^{*}$ or $\sigma_{A}=\iota_{A}^{*}$. 

\begin{lemma} \label{argmin sigma} Let $(X,\tau)$ be a LCS and
let $A\subset X$. Then $\sigma_{A}\ge0$ iff $0\in\operatorname*{cl}_{\tau}\operatorname*{conv}A$.
\end{lemma}

\begin{proof} Since $\sigma_{A}(0)=0$, using the biconjugate formula
we have
\[
\sigma_{A}\ge0\Leftrightarrow0\in\arg\min\sigma_{A}\Leftrightarrow0\in\partial\sigma_{A}(0)\Leftrightarrow\sigma_{A}(0)+\sigma_{A}^{*}(0)=0\Leftrightarrow
\]
\[
\iota_{\operatorname*{cl}_{\tau}\operatorname*{conv}A}(0)=0\Leftrightarrow0\in\operatorname*{cl}\,\!_{\tau}\operatorname*{conv}A.
\]
\end{proof}

\begin{proposition} Let $(X,\tau)$ be a LCS and let $T:X\rightrightarrows X^{*}$.
Then
\[
\Pr\nolimits _{X}(\operatorname*{dom}\varphi_{T})\subset\Pr\nolimits _{X}[\varphi_{T}=c]\cup\operatorname*{cl}\!\,_{\tau}\operatorname*{aff}D(T),\ \Pr\nolimits _{X^{*}}(\operatorname*{dom}\varphi_{T})\subset\Pr\nolimits _{X^{*}}[\varphi_{T}=c]\cup\operatorname*{cl}\!\,_{w^{*}}\operatorname*{aff}D(T).
\]
\end{proposition}

 \begin{proof} The conclusion holds if $\varphi_{T}$ is improper
(which is possible even when $T$ is non-void). That is why, we assume
that $\operatorname*{dom}\varphi_{T}\neq\emptyset$.

Let $z=(x,x^{*})\in\operatorname*{dom}\varphi_{T}$ with $x\not\in\operatorname*{cl}\!\,_{\tau}\operatorname*{aff}D(T)=:F$.
Fix $a_{0}\in F$ and denote by $S:=F-a_{0}$ the subspace parallel
to $F$. Since $x-a_{0}\not\in S=S^{\perp\perp}$, there is $u^{*}\in S^{\perp}$,
such that $\langle x-a_{0},u^{*}\rangle\neq0$. Note that $\varphi_{T}(x,x^{*}+tu^{*})=\varphi_{T}(x,x^{*})+t\langle a_{0},u^{*}\rangle$,
for every $t\in\mathbb{R}$. This yields that $(x,x^{*}+\gamma u^{*})\in[\varphi_{T}=c]$,
for $\gamma=(\varphi_{T}-c)(z)/\langle x-a_{0},u^{*}\rangle$. 

For the second inclusion we proceed similarly. \end{proof}

\begin{proposition} Let $(X,\tau)$ be a LCS and let $T:X\rightrightarrows X^{*}$.
If $z=(x,x^{*})\in\operatorname*{dom}\varphi_{T}$ then $\inf_{Z}(\varphi_{T}-c)=-\infty$
or $x\in\operatorname*{cl}_{\tau}\operatorname*{conv}D(T)$, $x^{*}\in\operatorname*{cl}\!\,_{w^{*}}\operatorname*{conv}R(T)$. 

In particular
\begin{equation}
\operatorname*{Pr}\!\,_{X}\operatorname*{dom}\varphi_{T}\subset\operatorname*{Pr}\!\,_{X}[\varphi_{T}\le c]\cup\operatorname*{cl}\!\,_{\tau}\operatorname*{conv}D(T),\ \operatorname*{Pr}\!\,_{X^{*}}\operatorname*{dom}\varphi_{T}\subset\operatorname*{Pr}\!\,_{X^{*}}[\varphi_{T}\le c]\cup\operatorname*{cl}\!\,_{w^{*}}\operatorname*{conv}R(T).\label{m9}
\end{equation}
\end{proposition}

\begin{proof} Assume that $z=(x,x^{*})\in\operatorname*{dom}\varphi_{T}$
and $\inf_{Z}(\varphi_{T}-c)>-\infty$. For every $u^{*}\in X^{*},$
pick $p=(0,u^{*})$ in (\ref{main}), noticing that $\sigma_{T-z}(0,u^{*})=\sigma_{D(T)-x}(u^{*})$. 

Taking into consideration that $c(p)=0$, we get 
\[
\forall u^{*}\in X^{*},\ \forall t\ge0,\ (\varphi_{T}-c)(z)+t\sigma_{D(T)-x}(u^{*})\ge\inf_{v^{*}\in X^{*}}(\varphi_{T}-c)(x,v^{*})\ge\inf_{Z}(\varphi_{T}-c)>-\infty.
\]
That yields $\sigma_{D(T)-x}\ge0$ from which, according to Lemma
\ref{argmin sigma}, $x\in\operatorname*{cl}_{\tau}\operatorname*{conv}D(T)$. 

Similarly, pick $p=(u,0)$, $u\in X$ to get $x^{*}\in\operatorname*{cl}\!\,_{w^{*}}\operatorname*{conv}R(T)$. 

In particular let $x\in\operatorname*{Pr}{}_{X}\operatorname*{dom}\varphi_{T}$,
that is, for some $x^{*}\in X^{*}$, $z=(x,x^{*})\in\operatorname*{dom}\varphi_{T}$.
If $x\not\in\operatorname*{cl}_{\tau}\operatorname*{conv}D(T)$ then
$\inf_{v^{*}\in X^{*}}(\varphi_{T}-c)(x,v^{*})=-\infty$. There exists
$v^{*}\in X^{*}$ such that $(\varphi_{T}-c)(x,v^{*})\le0$, that
is, $x\in\operatorname*{Pr}_{X}[\varphi_{T}\le c]$. 

The second inclusion in (\ref{m9}) follows a similar argument. \end{proof}

\begin{remark} The previous result is an extension of \cite[Lemma\ 3,\ p.\ 279]{MR3492118}.
More precisely we proved, in more detail, that if $z=(x,x^{*})\in\operatorname*{dom}\varphi_{T}$
then $\inf_{v^{*}\in X^{*}}(\varphi_{T}-c)(x,v^{*})=-\infty$ or $x\in\operatorname*{cl}_{\tau}\operatorname*{conv}D(T)$
and, by symmetry, $\inf_{v\in X}(\varphi_{T}-c)(v,x^{*})=-\infty$
or $x^{*}\in\operatorname*{cl}\!\,_{w^{*}}\operatorname*{conv}R(T)$.
\end{remark} 

\begin{corollary} Let $(X,\tau)$ be a LCS and let $T:X\rightrightarrows X^{*}$.
If $\inf_{Z}(\varphi_{T}-c)>-\infty$, in particular if $T$ is NI,
then $\operatorname*{Pr}\!\,_{X}\operatorname*{dom}\varphi_{T}\subset\operatorname*{cl}\!\,_{\tau}\operatorname*{conv}D(T)$,
$\operatorname*{Pr}\!\,_{X^{*}}\operatorname*{dom}\varphi_{T}\subset\operatorname*{cl}\!\,_{w^{*}}\operatorname*{conv}R(T)$.
\end{corollary}

\begin{theorem} Let $(X,\tau)$ be a LCS and let $T:X\rightrightarrows X^{*}$
be NI. Then
\begin{equation}
\begin{aligned}\forall z\in Z,\ (\varphi_{T}-c)(z) & \ge\sup\{-\frac{\sigma_{T-z}^{2}(p)}{4c(p)}\mid\sigma_{T-z}(p)<0\}\\
 & =\sup\{\frac{1}{4}\sigma_{T-z}^{2}(p)\mid c(p)=-1,\ \sigma_{T-z}(p)<0\};
\end{aligned}
\label{m3}
\end{equation}
\begin{equation}
\forall z,p\in Z,\ \sigma_{T-z}(p)+2\sqrt{(\varphi_{T}-c)(z)}\cdot\sqrt{|c(p)|}\ge0.\label{m4}
\end{equation}

\begin{equation}
\mathcal{M}(X)\ni[\varphi_{T}\le c]=[\varphi_{T}=c]\subset\operatorname*{cl}\!\,_{\tau\times w^{*}}\operatorname*{conv}(\operatorname*{Graph}T).\label{m6}
\end{equation}

In particular
\begin{equation}
\operatorname*{Pr}\!\,_{X}\operatorname*{dom}\varphi_{T}\subset\operatorname*{cl}\!\,_{\tau}\operatorname*{conv}D(T),\ \operatorname*{Pr}\!\,_{X^{*}}\operatorname*{dom}\varphi_{T}\subset\operatorname*{cl}\!\,_{w^{*}}\operatorname*{conv}R(T).\label{m8}
\end{equation}
\end{theorem} 

\begin{proof} The inequality (\ref{m3}) holds for $z\not\in\operatorname*{dom}\varphi_{T}$.
If $z\in\operatorname*{dom}\varphi_{T}$ and $\sigma_{T-z}(p)<0$
then $c(p)<0$ because $T$ is NI (see Remark \ref{r1}); in particular,
the right hand side of the inequality in (\ref{m3}) is well-defined.
According to (\ref{m2}), 
\[
(\varphi_{T}-c)(z)\ge\max_{t\ge0}\{t^{2}c(p)-t\sigma_{T-z}(p)\}=t_{V}^{2}c(p)-t_{V}\sigma_{T-z}(p)=-\frac{\sigma_{T-z}^{2}(p)}{4c(p)},
\]
where $t_{V}:=\frac{\sigma_{T-z}(p)}{2c(p)}>0$, and so (\ref{m3})
is proved. 

According to Remark \ref{r1}, more precisely from (\ref{eq:}), (\ref{eq:-1}),
relation (\ref{m4}) holds for $z\not\in\operatorname*{dom}\varphi_{T}$
or if $\sigma_{T-z}(p)\ge0$ or $c(p)\ge0$. Assume that $z\in\operatorname*{dom}\varphi_{T}$
and $\sigma_{T-z}(p)<0$, from which, again $c(p)<0$. 

The inequality in (\ref{m3}) provides $-4c(p)(\varphi_{T}-c)(z)\ge\sigma_{T-z}^{2}(p)$,
$2\sqrt{-c(p)}\cdot\sqrt{(\varphi_{T}-c)(z)}\ge|\sigma_{T-z}(p)|=-\sigma_{T-z}(p)$
and (\ref{m4}) is proved. 

The inclusion in (\ref{m6}) is a direct consequence of (\ref{m4})
and Lemma \ref{argmin sigma} while (\ref{m8}) follows from (\ref{m6})
and (\ref{m9}). \end{proof}

\begin{remark} \label{r2} When $T$ is monotone and NI, (\ref{m8})
is equivalent to
\begin{equation}
\operatorname*{cl}\!\,_{\tau}\operatorname*{Pr}\!\,_{X}\operatorname*{dom}\varphi_{T}=\operatorname*{cl}\!\,_{\tau}\operatorname*{conv}D(T),\ \operatorname*{cl}\!\,_{w^{*}}\operatorname*{Pr}\!\,_{X^{*}}\operatorname*{dom}\varphi_{T}=\operatorname*{cl}\!\,_{w^{*}}\operatorname*{conv}R(T),\label{eq:-6}
\end{equation}
because in this case $\operatorname*{Graph}T\subset[\varphi_{T}=c]$
(see e.g. \cite[Theorem\ 3.4,\ p.\ 61]{MR1009594}), so $D(T)\subset\operatorname*{Pr}_{X}\operatorname*{dom}\varphi_{T}$,
$R(T)\subset\operatorname*{Pr}_{X^{*}}\operatorname*{dom}\varphi_{T}$. 

The equlities in (\ref{eq:-6}) are a subtle extension of \cite[Lemma\ 3,\ (7),\ p.\ 279]{MR3492118}.
Recall that for $T$ a monotone NI operator, $[\varphi_{T}=c]$ is
the only maximal monotone extension of $T$ and $\varphi_{T}=\varphi_{[\varphi_{T}=c]}$
(see \cite[Proposition\ 4,\ p.\ 35]{MR2594359}). Under these assumptions
(\ref{eq:-6}) implies
\begin{equation}
\operatorname*{cl}\!\,_{\tau}\operatorname*{conv}D(T)=\operatorname*{cl}\!\,_{\tau}\operatorname*{Pr}\!\,_{X}\operatorname*{dom}\varphi_{T}=\operatorname*{cl}\!\,_{\tau}\operatorname*{conv}\operatorname*{Pr}\!\,_{X}[\varphi_{T}=c].\label{eq:-7}
\end{equation}

In general, even if $T$ is NI, its domain $D(T)$ is not necessarily
a subset of $\operatorname*{Pr}_{X}\operatorname*{dom}\varphi_{T}$.
For example $T:D(T)=\mathbb{R}\rightrightarrows\mathbb{R}$, $T(x)=\{0\}$,
if $x\neq0$; $T(0)=\mathbb{R}$, namely, $\operatorname*{Graph}T$
is the union of the $x$ and $y-$axes in the plane, has $\varphi_{T}=\iota_{\{(0,0)\}}\ge c$;
whence $D(T)=\mathbb{R}\not\subset\operatorname*{Pr}_{X}\operatorname*{dom}\varphi_{T}=\{0\}$.
\end{remark}

When $(X,\|\cdot\|)$ is a normed space and $T:X\rightrightarrows X^{*}$
we can see $T:X^{**}\rightrightarrows X^{*}$ and $T^{-1}:X^{*}\rightrightarrows X^{**}$
through the canonical injection $J:X\rightarrow X^{\ast\ast}$. More
precisely, 
\[
x^{*}\in T(x^{**})\Leftrightarrow\exists x\in X,\ x^{**}=Jx\ {\rm and}\ x^{*}\in T(x).
\]
In the next result $(X^{**}\times X^{*},w^{*})$ denotes the space
$X^{**}\times X^{*}$ endowed with its weak-star topology (that comes
from the duality$(X^{**}\times X^{*},X^{*}\times X)$); while for
$z=(x,x^{*})\in Z=X\times X^{*}$, $z^{T}:=(x^{*},x)\in Z^{*}=X^{*}\times X^{**}$. 

\begin{proposition} \label{est-norm} Let $(X,\|\cdot\|)$ be a normed
space and let $T:X\rightrightarrows X^{*}$ be NI. Then
\begin{equation}
\forall z\in Z,\ (\varphi_{T}-c)(z)\ge\frac{1}{2}{\rm dist}^{2}(z,\operatorname*{cl}\!\,_{(X^{**}\times X^{*},w^{*})}\operatorname*{conv}\operatorname*{Graph}T),\label{m5}
\end{equation}
where for $z\in X^{**}\times X^{*}$, $S\subset X^{**}\times X^{*}$,
${\rm dist}(z,S)=\inf\{\|z-u\|_{X^{**}\times X^{*}}\mid u\in S\}$
and $\|(x^{**},x^{*})\|_{X^{**}\times X^{*}}=\sqrt{\|x^{**}\|^{2}+\|x^{*}\|^{2}}$,
$x^{**}\in X^{**}$, $x^{*}\in X^{*}$. 

If, in addition, $X$ is reflexive then 
\begin{equation}
\forall z\in Z,\ (\varphi_{T}-c)(z)\ge\frac{1}{2}{\rm dist}^{2}(z,\operatorname*{conv}\operatorname*{Graph}T),\label{m7}
\end{equation}

\end{proposition}

\begin{proof} Note that for every $p\in Z$, $\sqrt{|c(p)|}\le(\sqrt{2})^{-1}\|p\|$
so, for $\alpha:=(\varphi_{T}-c)(z)$, (\ref{m4}) reads
\[
\forall p\in Z,\ \sigma_{T-z}(p)+\sqrt{2\alpha}\cdot\|p\|\ge0.
\]
As previously seen, $0\in\arg\min(\sigma_{T-z}+\sqrt{2\alpha}\|\cdot\|)$,
$0\in\partial(\sigma_{T-z}+\sqrt{2\alpha}\|\cdot\|)(0)=\partial\sigma_{T-z}(0)+\sqrt{2\alpha}B^{*}$
(see e.g \cite[Theorem\ 2.8.7\ (iii),\ p.\ 126]{MR1921556}), where
$B^{*}$ is the unit ball in $Z^{*}$ and ``$\partial$'' is the
subdifferential considered with respect to the dual system $(Z,Z^{*})$. 

Hence there is $q\in\sqrt{2\alpha}B^{*}$, $-q\in\partial\sigma_{T-z}(0)$,
that is, $-q\in\operatorname*{cl}_{(Z^{*},w^{*})}\operatorname*{conv}\operatorname*{Graph}T^{-1}-z^{T}$. 

We have $z-q^{T}\in\operatorname*{cl}_{(X^{**}\times X^{*},w^{*})}\operatorname*{conv}\operatorname*{Graph}T$,
$\|q\|_{Z^{*}}=\|q^{T}\|_{X^{**}\times X^{*}}\le\sqrt{2\alpha}$,
\[
{\rm dist}(z,\operatorname*{cl}\!\,_{(X^{**}\times X^{*},w^{*})}\operatorname*{conv}\operatorname*{Graph}T)\le\|z-(z-q^{T})\|_{X^{**}\times X^{*}}=\|q\|_{Z^{*}}\le\sqrt{2\alpha}.
\]
\end{proof}

\begin{remark} The inequalities in Proposition \ref{est-norm} give
rough estimates of how far $\varphi_{T}$ is from $c$, since there
are maximal monotone operators that cover the whole space $Z$ with
their convex hull such as $f(x)=x^{3}$, $x\in\mathbb{R}$; $\operatorname*{conv}(\operatorname*{Graph}f)=\mathbb{R}^{2}$.
However, such operators $T:X\rightrightarrows X^{*}$, defined in
a Banach space $X$, that have $\operatorname*{conv}(\operatorname*{Graph}T)=X\times X^{*}$
also satisfy $\operatorname*{conv}D(T)=X$, in which case $D(T)=X$
(see \cite[Theorem\ 1,\ p.\ 398]{MR0253014}) and they are well behaved
and understood (see f.i. \cite{MR3252437,MR3492118,MR2577332}). \end{remark}

\begin{remark} If we change the norms on $Z$ and $Z^{*}$ to a weighted
norms of the form $\|(x,x^{*})\|_{\delta}=\sqrt{\delta||x\|^{2}+\frac{1}{\delta}||x^{*}\|^{2}}$,
$x\in X$, $x^{*}\in X^{*}$, where $\delta>0$, we get that, for
every $z\in\operatorname*{dom}\varphi_{T}$, $\delta>0$, $\exists(x_{\delta}^{**},x_{\delta}^{*})\in\operatorname*{cl}\!\,_{(X^{**}\times X^{*},w^{*})}\operatorname*{conv}\operatorname*{Graph}T)$
such that $\|z-(x_{\delta}^{**},x_{\delta}^{*})\|_{\delta}\le(\varphi_{T}-c)(z)$.
For $\delta\to0$, $x_{\delta}^{**}\to x$ strongly in $X^{**}$ and
so we recover again (\ref{m8}). \end{remark}

\section{Estimates for monotone operators}

Recall that $T$ is monotone iff $\operatorname*{Graph}T\subset[\varphi_{T}\le c]$
while $T$ is maximal monotone iff $\operatorname*{Graph}T=[\varphi_{T}\le c]$.
Also, recall the identity
\begin{equation}
\forall z,w\in Z,\ 0\le t\le1,\ c(tz+(1-t)w)=tc(w)+(1-t)c(w)-t(1-t)c(z-w).\label{i1}
\end{equation}
 Due to the convexity of $\varphi_{T}$, from (\ref{i1}) we have
\begin{equation}
\forall z,w\in Z,\ 0\le t\le1,\ (\varphi_{T}-c)(tz+(1-t)w)\le t(\varphi_{T}-c)(z)+(1-t)(\varphi_{T}-c)(w)+t(1-t)c(z-w).\label{i2}
\end{equation}

Recall also that $T^{+}:X\rightrightarrows X^{*}$ is the operator
defined by $\operatorname*{Graph}(T^{+}):=[\varphi_{T}\le c]$ or,
equivalently, that $z\in\operatorname*{Graph}(T^{+})$ iff $z$ is
monotonically related to $T$; and that for $T:X\rightrightarrows X^{*}$
monotone, $T\subset T^{++}\subset T^{+}$, where $T^{++}:=(T^{+})^{+}$.
For more properties of the application $T\to T^{+}$ see \cite{MR2128696,MR1381386,MR2594359}. 

\medskip

The next result first two subpoints can be seen as an extension of
the known fact that, for a maximal monotone operator $T:X\rightrightarrows X^{*}$,
for every $z\in[\varphi_{T}>c]$ there is $w\in T$ such that $c(z-w)<0$;
while the third subpoint explores the structure of the set $[\varphi_{T}<c]$. 

\begin{proposition} Let $X$ be a LCS and let $T:X\rightrightarrows X^{*}$
be monotone. Then

\medskip{}

\emph{(i)} for every $z\in[\varphi_{T}>c]$ there is $w\in[\varphi_{T}\le c]$
such that $c(z-w)<0$;

\medskip{}

\emph{(ii)} for every $z\in[\varphi_{T}>c]\cap\operatorname*{dom}\varphi_{T}$
there is $w\in[\varphi_{T}=c]$ such that $c(z-w)<0$;

\medskip

\emph{(iii)} for every $z\in[\varphi_{T}<c]$ and for every $t\in(0,1)$
there is $w\in T$ such that $tz+(1-t)w\in[\varphi_{T}<c]$. \end{proposition}

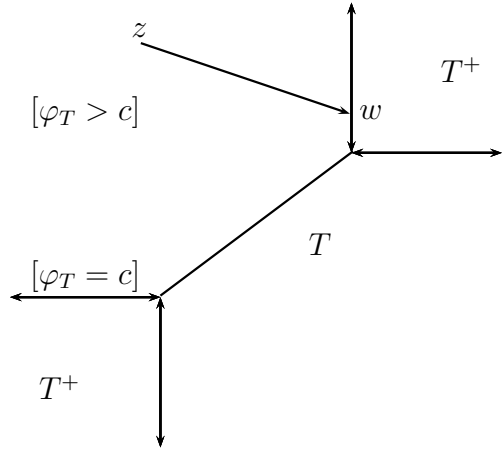
\centerline{\begin{pspicture}(0,-2.96)(6.54,2.96)
\psline[linewidth=0.03cm](2.0,-0.94)(4.56,0.98)
\rput(4.54,0.96){\psaxes[linewidth=0.03,arrowsize=0.05291667cm 2.0,arrowlength=1.4,arrowinset=0.4,labels=none,ticks=none,ticksize=0.10583333cm]{<->}(0,0)(0,0)(2,2)} \rput(2.0,-0.96){\psaxes[linewidth=0.03,arrowsize=0.05291667cm 2.0,arrowlength=1.4,arrowinset=0.4,labels=none,ticks=none,ticksize=0.10583333cm]{<->}(0,0)(-2,-2)(0,0)}
\rput(4.121406,-0.255){$T$}
\rput(6.001406,2.065){$T^+$}
\rput(0.6614063,-2.155){$T^+$}
\rput(1,-0.675){$[\varphi_T=c]$}
\rput(1,1.5){$[\varphi_T>c]$}
\psline[linewidth=0.03cm,arrowsize=0.05291667cm 2.0,arrowlength=1.4,arrowinset=0.4]{->}(1.74,2.42)(4.52,1.48)
\rput(1.7214062,2.6){$z$} \usefont{T1}{ptm}{m}{n} 
\rput(4.8,1.505){$w$} 
\end{pspicture} }

\begin{proof} Since $T\in\mathcal{M}(X)$ we know that $T\subset T^{+}:=[\varphi_{T}\le c]$
so $T^{+}\supset T^{++}:=(T^{+})^{+}$ . Subpoint (i) states that
$Z\setminus T^{+}\subset Z\setminus T^{++}$. 

(ii) If $z\in[\varphi_{T}>c]\cap\operatorname*{dom}\varphi_{T}$ then,
according to (i), there is $u\in T^{+}$ such that $c(z-u)<0$. The
function $f(t):=(\varphi_{T}-c)(z+t(u-z))$, $t\in[0,1]$, is continuous
on $[0,1]$, $f(0)>0$, $f(1)\le0$. Hence there is $0<s\le1$, $s\in[f=0]$,
i.e., $w:=z+s(u-z)\in[\varphi_{T}=c]$ and $c(z-w)=s^{2}c(z-u)<0$.

For (iii) first note that for $w\in T\subset[\varphi_{T}\le c]$ (\ref{i2})
implies
\begin{equation}
\forall z\in Z,\ 0\le t\le1,\ (\varphi_{T}-c)(tz+(1-t)w)\le t[(\varphi_{T}-c)(z)+(1-t)c(z-w)].\label{i3}
\end{equation}
Assume by contradiction that there are $z\in[\varphi_{T}<c]$ and
$t\in(0,1)$ such that for every $w\in T$, $(\varphi_{T}-c)(tz+(1-t)w)\ge0$.
From (\ref{i3}) we get $(\varphi_{T}-c)(z)+(1-t)c(z-w)\ge0$. Pass
to $\inf_{w\in T}$, taking into account that $\inf_{w\in T}c(z-w)=(c-\varphi_{T})(z)\in\mathbb{R}$,
to get the contradiction $t(\varphi_{T}-c)(z)\ge0$. \end{proof}

\begin{remark} In contrapositive form, subpoint (iii) of the previous
proposition spells: -if there exists $t\in(0,1)$ such that,  for
every $w\in T$, $tz+(1-t)w\in[\varphi_{T}\ge c]$ then $z\in[\varphi_{T}\ge c]$.
\end{remark}

Since every maximal monotone operator is NI (see \cite[Theorems\ 3.4,\ 3.8]{MR1009594})
we know from (\ref{m8}) that in a LCS $(X,\tau)$
\begin{equation}
T\in\mathfrak{M}(X)\Longrightarrow\operatorname*{Pr}\!\,_{X}\operatorname*{dom}\varphi_{T}\subset\operatorname*{cl}\!\,_{\tau}\operatorname*{conv}D(T).\label{i4}
\end{equation}
We prove (\ref{i4}) using a minmax technique. 

\strut

\begin{proof}[Proof of (21)] Let $x\in\operatorname*{Pr}{}_{X}\operatorname*{dom}\varphi_{T}$,
that is, $(x,x^{*})\in\operatorname*{dom}\varphi_{T}$, for some $x^{*}\in X^{*}$. 

Assume that $x\not\in D(T)$. Since $T$ is maximal monotone, for
every $u^{*}\in X^{*}$, 
\[
\varphi_{T}(x,x^{*}+u^{*})>\langle x,x^{*}+u^{*}\rangle.
\]
Hence there exists $(\bar{a},\bar{a}^{*})\in T$ (which depends on
$u^{*}$) such that
\[
\langle x-\bar{a},\bar{a}^{*}\rangle+\langle\bar{a},x^{*}+u^{*}\rangle>\langle x,x^{*}+u^{*}\rangle\Rightarrow
\]
\begin{equation}
\begin{aligned}\varphi_{T}(x,x^{*}) & \ge\langle x-\bar{a},\bar{a}^{*}\rangle+\langle\bar{a},x^{*}\rangle\\
 & >\langle x,x^{*}\rangle+\langle x-\bar{a},u^{*}\rangle\\
 & \ge\langle x,x^{*}\rangle+\inf_{a\in D(T)}\langle x-a,u^{*}\rangle\\
 & =\langle x,x^{*}\rangle+\inf_{a\in\operatorname*{conv}D(T)}\langle x-a,u^{*}\rangle.
\end{aligned}
\label{i5}
\end{equation}
Let $V$ be a generic closed convex $\tau-$neighborhood of $0\in X$.
According to Bourbaki's Theorem, its polar $V^{\circ}$ is convex
weak-star compact. Pass to supremum over $u^{*}\in V^{\circ}$ in
(\ref{i5}) to find 
\[
\sup_{u^{*}\in V^{\circ}}\inf_{a\in\operatorname*{conv}D(T)}\langle x-a,u^{*}\rangle\le(\varphi_{T}-c)(x,x^{*}).
\]

Consider $f:V^{\circ}\times\operatorname*{conv}D(T)\to\mathbb{R}$,
$f(u^{*},a):=\langle x-a,u^{*}\rangle$. Then, for every $a\in\operatorname*{conv}D(T)$,
$f(\cdot,a)$ is linear and weak-star continuous while, for every
$u^{*}\in V^{\circ}$, $f(u^{*},\cdot)$ is affine and $\tau-$continuous.
We employ the minmax theorem (see e.g. \cite[Theorem\ 2.10.2,\ p.\  144]{MR1921556})
to find that, for every $V$ a closed convex $\tau-$neighborhood
of $0\in X$,
\[
\sup_{u^{*}\in V^{\circ}}\inf_{a\in\operatorname*{conv}D(T)}\langle x-a,u^{*}\rangle=\inf_{a\in\operatorname*{conv}D(T)}\sup_{u^{*}\in V^{\circ}}\langle x-a,u^{*}\rangle\le(\varphi_{T}-c)(x,x^{*})<+\infty.
\]
Change $V$ with $kV$ for $k>0$ and let $k\to0$ to get that, $\inf_{a\in\operatorname*{conv}D(T)}\sup_{u^{*}\in V^{\circ}}\langle x-a,u^{*}\rangle=0$.
In particular $\inf_{a\in\operatorname*{conv}D(T)}\sup_{u^{*}\in V^{\circ}}\langle x-a,u^{*}\rangle<1$
so there is $a\in\operatorname*{conv}D(T)$ such that $\sup_{u^{*}\in V^{\circ}}\langle x-a,u^{*}\rangle<1$,
i.e., since $V^{\circ}$ is symmetric, $x-a\in V^{\circ\circ}=V$.
Hence, for every $\tau-$closed convex neighborhood $V$ of $0\in X$,
$x\in\operatorname*{conv}D(T)+V$; that is, $x\in\operatorname*{cl}\!\,_{\tau}\operatorname*{conv}D(T)$.
\end{proof}

\begin{lemma} Let $X$ be a LCS and let $T:X\rightrightarrows X^{*}$.
Then $T$ is monotone iff $T^{+}$ is NI. \end{lemma}

\begin{proof} Assume that $T$ is monotone and let $M$ be a maximal
monotone extension of $T$. Then $M\subset T^{+}$, $\varphi_{T^{+}}\ge\varphi_{M}\ge c$. 

Conversely, if $\varphi_{T^{+}}\ge c$ then $T\subset[\varphi_{T^{+}}\le c]=[\varphi_{T^{+}}=c]\in\mathcal{M}(X)$.
\end{proof}

\begin{proposition} Let $(X,\tau)$ be a LCS and let $T:X\rightrightarrows X^{*}$
be monotone. Then 
\begin{equation}
\operatorname*{Pr}\!\,_{X}\operatorname*{dom}\varphi_{T}\subset D(T^{+})\cup\operatorname*{cl}\!\,_{\tau}\operatorname*{Pr}\!\,_{X}\operatorname*{dom}\varphi_{T^{+}}\subset\operatorname*{cl}\!\,_{\tau}\operatorname*{conv}D(T^{+})\subset\operatorname*{cl}\!\,_{\tau}\operatorname*{Pr}\!\,_{X}\operatorname*{dom}\varphi_{T},\label{eq:-3}
\end{equation}
\begin{equation}
\operatorname*{cl}\!\,_{\tau}\operatorname*{conv}D(T^{+})=\operatorname*{cl}\!\,_{\tau}\operatorname*{Pr}\!\,_{X}\operatorname*{dom}\varphi_{T}.\label{eq:-4}
\end{equation}
\end{proposition} 

\begin{proof} The first inclusion in (\ref{eq:-3}) follows from
(\ref{m9}) and $T\subset T^{++}=[\varphi_{T^{+}}\le c]\subset\operatorname*{dom}\varphi_{T^{+}}$.
The second inclusion in (\ref{eq:-3}) is a consequence of (\ref{m8})
applied for $T^{+}$ while the third inclusion in (\ref{eq:-3}) is
derived from $T^{+}\subset\operatorname*{dom}\varphi_{T}$.

Relation (\ref{eq:-4}) is a direct consequence of (\ref{eq:-3})
or of (\ref{m9}) and $T\subset T^{+}\subset\operatorname*{dom}\varphi_{T}$.
\end{proof}

\eject

\end{document}